\documentclass[12pt]{article}
\usepackage[margin=1in]{geometry}
\usepackage[hidelinks]{hyperref}

\usepackage{amsmath,amsthm,amssymb}
\usepackage{mathtools}
\usepackage{tikz}

\theoremstyle{plain}
\newcounter{documentcounter}
\newtheorem{theorem}[documentcounter]{Theorem}
\newtheorem{proposition}[documentcounter]{Proposition}
\newtheorem{corollary}[documentcounter]{Corollary}
\newtheorem{lemma}[documentcounter]{Lemma}
\newtheorem{innercustomthm}{Theorem}
\newenvironment{customthm}[1]
  {\renewcommand\theinnercustomthm{#1}\innercustomthm}
  {\endinnercustomthm}

\theoremstyle{definition}
\newtheorem{definition}[documentcounter]{Definition}
\newtheorem*{definition*}{Definition}
\newtheorem{example}[documentcounter]{Example}
\newtheorem{problem}[documentcounter]{Problem}

\newcommand{\defn}[1]{\textbf{#1}}

\newcommand{\script}[1]{\mathcal{#1}}         
\newcommand{\defeq}{\coloneqq}                
\newcommand{\set}[2]{\left\{{#1}\,:\,{#2}\right\}} 
\newcommand{\card}[1]{\left|#1\right|}        
\newcommand{\ncard}[1]{\#{#1}}
\newcommand{\union}{\cup}                     
\newcommand{\intersect}{\cap}                 
\newcommand{\Union}{\bigcup}                  
\newcommand{\Intersect}{\bigcap}              
\newcommand{\ZZ}{\mathbb{Z}}                  
\newcommand{\RR}{\mathbb{R}}                  

\newcommand{\rcov}{r_{\textrm{c}}}            
\newcommand{\meet}{\wedge}                    
\newcommand{\join}{\vee}                      

\DeclareMathOperator{\cl}{cl}                 
\newcommand{\indep}{\script{I}}               
\newcommand{\bases}{\script{B}}               
\newcommand{\circuits}{\script{C}}            
\newcommand{\flats}{\script{F}}               

\DeclareMathOperator{\Act}{Act}               
\DeclareMathOperator{\EA}{EA}                 
\DeclareMathOperator{\EP}{EP}                 
\DeclareMathOperator{\IP}{IP}                 
\DeclareMathOperator{\lex}{lex}               

\newcommand{\clfn}{\sigma}                    
\DeclareMathOperator{\ex}{ex}                 
\newcommand{\K}{{\script{K}}}                 
\newcommand{\F}{\script{F}}                   
\newcommand{\trace}[2]{{{#1}\!:\!{#2}}}       

\DeclareMathOperator{\Ext}{Ext}               
\newcommand{\extDec}{_{\Ext}}                 
\newcommand{\clExt}{\clfn\extDec}             
\newcommand{\exExt}{\ex\extDec}               
\DeclareMathOperator{\leqext}{\leq\extDec}    

\begin{document}

\title{Convexity in ordered matroids \\ and the generalized external order}

\author{Bryan R.\ Gillespie\footnote{Email: \href{mailto:Bryan.Gillespie@colostate.edu}{\texttt{Bryan.Gillespie@colostate.edu}}}\\[1ex] Department of Mathematics\\ Colorado State University\\ Fort Collins, CO, U.S.A.}

\date{Aug 21, 2020}

\maketitle

\begin{abstract}
  In 1980, Las Vergnas defined a notion of discrete convexity for oriented matroids, which Edelman subsequently related to the theory of anti-exchange closure functions and convex geometries.  In this paper, we use generalized matroid activity to construct a convex geometry associated with an ordered, \emph{unoriented} matroid.  The construction in particular yields a new type of representability for an ordered matroid defined by the affine representability of its corresponding convex geometry.

  The lattice of convex sets of this convex geometry induces an ordering on the matroid independent sets which extends the external active order on matroid bases.  We show that this generalized external order forms a supersolvable meet-distributive lattice refining the geometric lattice of flats, and we uniquely characterize the lattices isomorphic to the external order of a matroid.  Finally, we introduce a new trivariate generating function generalizing the matroid Tutte polynomial.
\end{abstract}

\section{Introduction}

An \emph{anti-exchange closure function} on a finite set gives a discrete analog of the classical ``convex hull'' closure function in $\RR^n$.  This gives rise to the notion of a \emph{convex geometry}, which was independently discovered by Edelman \cite{edelman_meet-distributive_1980} and Jamison \cite{jamison_perspective_convexity_1982}.
In \cite{las_vergnas_convexity_1980}, Las Vergnas introduced an anti-exchange closure function associated with an oriented matroid, which was related to the theory of convex geometries by Edelman \cite{edelman_lattice_1982}.

In this work, we describe a parallel theory of convexity for unoriented matroids using a closure function related to the \emph{active orders} on matroid bases.
These orders were originally studied by Bj\"orner \cite{bjorner_homology_preprint_1979,bjorner_homology_1992} in relation to lexicographic shellings of matroid independence complexes, where the inclusion ordering of \emph{restriction sets} gives an ordering on the bases of a matroid, the \emph{internal active order}.
Additional work by Dawson \cite{dawson_collection_1984} and Purtill \cite[Ex.~7.5.c]{bjorner_homology_1992} studied the greedoid structure of these collections of restriction sets, and Las Vergnas \cite{las_vergnas_active_2001} showed that the internal order and its dual the \emph{external active order} admit a natural lattice structure by respectively extending with a formal maximal and minimal element.
The convexity theory developed in this work induces a lattice structure on the independent sets of an ordered matroid that canonically extends the external order on matroid bases.

Motivation for this work comes from the theory of \emph{zonotopal algebras} of Holtz and Ron \cite{holtz_zonotopal_2011}, where the external order has fundamental connections with the structure of the polynomial \emph{zonotopal spaces} of a representable matroid.
One example of this is the \emph{forward exchange matroids} of Lenz \cite{lenz_zonotopal_2016}, which provide a generalized setting for the construction of zonotopal spaces; here, the defining \emph{forward exchange property} of a collection of matroid bases is equivalent to the property of being an upper order ideal in $\Ext(M)$.
As another application, in \cite{gillespie_generalized_2018} the author uses the external order to describe the differential properties of Lenz's canonical basis of the zonotopal $\script{D}$-space \cite{lenz_zonotopal_2016}, which in particular leads to a direct algorithm to compute the polynomials in this basis.

We next summarize our main constructions and results.
The \emph{external active closure function} $\clExt$ is defined in \cite{las_vergnas_active_2001} in terms of the following generalization of the classical notion of matroid activity,
first introduced in dual form by Dawson \cite{dawson_collection_1984}.\footnote{See also \cite{bari_chromatic_polynomials_1979,gordon_generalized_1990}.}
Let $M = (E, \indep)$ be an ordered matroid with ground set $E$ and independent sets $\indep$, and let $A \subseteq E$.  An element $x \in E$ is called \emph{active} with respect to $A$ if there is a circuit $C \subseteq A \union \{x\}$ with $x = \min(C)$.
The elements of $E \setminus A$ that are active with respect to $A$ are called \emph{externally active}, and are denoted by $\EA(A)$.
The external active closure function is then defined by
\[
  \clExt(A) \defeq A \union \EA(A).
\]

Recalling that any closure function $\clfn$ has a corresponding \emph{extreme point function} given by $\ex(A) = \set{x \in A}{x \notin \clfn(A \setminus x)}$, we also define the \emph{external active extreme point function} $\exExt$ by
\[
  \exExt(A) \defeq {\max}_{\lex} \set{I \in \indep}{I \subseteq A \text{ and } \cl(I) = \cl(A)}.
\]
Our first result relates $\clExt$ and $\exExt$ to the theory of convex geometries as follows.

\begin{theorem}
  \label{thm:ext-closure-anti-exchange}
  The map $\clExt$ is an anti-exchange closure function, and its corresponding extreme point function is $\exExt$.
\end{theorem}

The identification of an anti-exchange closure function $\clExt$ and its associated convex geometry with any ordered matroid gives rise to a potentially interesting new kind of representability.  A convex geometry is called \emph{affine} if it may be derived from a standard anti-exchange closure function on points in Euclidean space (see Section \ref{subsec:background-conv-geoms}), and the question of classifying affine convex geometries in general is open, and thought to be difficult.  In the setting of the convex geometries assocated with ordered matroids, the question could be more tractible: which ordered matroids are \emph{convex representable}, in the sense that their associated convex geometry is affine?  For a brief discussion and additional references, see Section \ref{subsec:future-work-conv-geoms}.

Theorem \ref{thm:ext-closure-anti-exchange} in particular implies that the closed sets of $\clExt$ are in one-to-one correspondence with the independent sets of $M$, and we define the \emph{external order} $\Ext(M)$ on $\indep$ by
\[
  I \leqext J \text{ if and only if } \clExt(I) \subseteq \clExt(J).
\]
This order is a \emph{meet-distributive lattice} which can be seen to extend the external active order on bases by comparison with \cite{las_vergnas_active_2001} Proposition 3.1, omitting the formal minimal element added in that work.  In this lattice, the external active order on bases embeds as an upper order ideal, and the lattice meet satisfies $\clExt(I \meet J) = \clExt(I) \intersect \clExt(J)$.  We additionally show in Proposition \ref{prop:ext-order-others} that the external order is a refinement of the geometric lattice of flats of a matroid by the map $I \mapsto \cl(I)$, and is consistent with inclusion of independent sets and $\leq^*$-lexicographic order.

Figure \ref{fig:example-ext-order} gives an example of the external order associated with the ordered linear matroid represented by the columns of the matrix
\[
  X =
  \begin{bmatrix}
    1 & 0 & 1 & 1 \\
    0 & 1 & 1 & 0
  \end{bmatrix},
\]
where the numbers 1 through 4 indicate the column number, labeled from left to right.

\begin{figure}
  \centering
  \begin{minipage}[c]{.3\textwidth}
    \centering
    \begin{tikzpicture}[yscale=1, xscale=0.75]
      \tikzstyle{feas set style}=[inner sep=1mm, shape=rectangle]
      \node[feas set style] (p34) at ( 0,4) {1234};
      \node[feas set style] (p23) at (-1,3) {123};
      \node[feas set style] (p24) at ( 1,3) {124};
      \node[feas set style] (p13) at (-2,2) {13};
      \node[feas set style] (p12) at ( 0,2) {12};
      \node[feas set style] (pp4) at ( 2,2) {14};
      \node[feas set style] (pp3) at (-2,1) {3};
      \node[feas set style] (pp1) at ( 0,1) {1};
      \node[feas set style] (pp2) at ( 2,1) {2};
      \node[feas set style] (ppp) at ( 0,0) {$\emptyset$};
      \draw (p34) -- (p23);
      \draw (p34) -- (p24);
      \draw (p23) -- (p13);
      \draw (p23) -- (p12);
      \draw (p24) -- (p12);
      \draw (p24) -- (pp4);
      \draw (p13) -- (pp3);
      \draw (p13) -- (pp1);
      \draw (p12) -- (pp1);
      \draw (p12) -- (pp2);
      \draw (pp4) -- (pp1);
      \draw (pp3) -- (ppp);
      \draw (pp1) -- (ppp);
      \draw (pp2) -- (ppp);
      \node at (0,4.8) {Convex Sets};
    \end{tikzpicture}
  \end{minipage}
  \begin{minipage}[c]{.3\textwidth}
    \centering
    \begin{tikzpicture}[yscale=1, xscale=0.75]
      \tikzstyle{feas set style}=[inner sep=1mm, shape=rectangle]
      \node[feas set style] (p34) at ( 0,4) {\textbf{34}};
      \node[feas set style] (p23) at (-1,3) {\textbf{23}};
      \node[feas set style] (p24) at ( 1,3) {\textbf{24}};
      \node[feas set style] (p13) at (-2,2) {\textbf{13}};
      \node[feas set style] (p12) at ( 0,2) {\textbf{12}};
      \node[feas set style] (pp4) at ( 2,2) {4};
      \node[feas set style] (pp3) at (-2,1) {3};
      \node[feas set style] (pp1) at ( 0,1) {1};
      \node[feas set style] (pp2) at ( 2,1) {2};
      \node[feas set style] (ppp) at ( 0,0) {$\emptyset$};
      \draw[thick] (p34) -- (p23);
      \draw[thick] (p34) -- (p24);
      \draw[thick] (p23) -- (p13);
      \draw[thick] (p23) -- (p12);
      \draw[thick] (p24) -- (p12);
      \draw (p24) -- (pp4);
      \draw (p13) -- (pp3);
      \draw (p13) -- (pp1);
      \draw (p12) -- (pp1);
      \draw (p12) -- (pp2);
      \draw (pp4) -- (pp1);
      \draw (pp3) -- (ppp);
      \draw (pp1) -- (ppp);
      \draw (pp2) -- (ppp);
      \node at (0,4.8) {Independent Sets};
    \end{tikzpicture}
  \end{minipage}
  \caption{Lattice of convex sets of the external active closure function $\clExt$, and the corresponding external order $\leqext$ on independent sets; the external order on matroid bases appears in bold.}
  \label{fig:example-ext-order}
\end{figure}

We develop the following characterization of the lattices derived from the external order.

\begin{theorem}
  \label{thm:ext-order-char}
  A finite lattice $L$ is isomorphic to the external order $\leqext$ of an ordered matroid if and only if it is meet-distributive, supersolvable, and has increasing and submodular covering rank function.
\end{theorem}

Here, the the covering rank function of a lattice $L$ is defined for $x \in L$ as the number of elements in $L$ covered by $x$.  The condition of lattice supersolvability in particular arises because of the need for a certain ordering consistency with respect to the ground set of the underlying matroid (see Example \ref{ex:matroidal-not-ext-order});
it is derived by relating the external order with the theory of \emph{supersolvable antimatroids} developed by Armstrong \cite{armstrong_sorting_2009}.

The origin of the active orders in the theory of matroid shellings suggests the potential for connections with the area of topological combinatorics.  One result, given by Proposition \ref{prop:bc-complex-embedding}, is that the \emph{broken circuit complex} of an ordered matroid is embedded in $\Ext(M)$ as a lower order ideal, with facets given by the $\leqext$-minimal bases of $M$.
An interesting application following from results for general convex geometries is the following partition of the Boolean lattice, which refines the partition of Crapo \cite{crapo_tutte_1969} over matroid bases.
\begin{theorem}
  \label{thm:crapo-type-partition}
  Let $M$ be an ordered matroid.  Then the intervals
  \[
    [I, I \union \EA(I)], \hspace{2mm} I \in \indep(M)
  \]
  form a partition of the Boolean lattice $2^{E(M)}$, and this partition is a refinement of the classical partition of Crapo.
\end{theorem}
The first statement above was derived in complemented form by Dawson \cite{dawson_collection_1984}, but we give an original proof here, as well as providing the additional result that the partition refines that of Crapo.  Motivated by this partition result, we define a new trivariate generating function generalizing the matroid Tutte polynomial, the \emph{external structure polynomial} $P_M$ of $M$, by
\[
  P_M(x, y, z) \defeq \sum_{I \in \indep(M)} x^{r(M) - \card{I}} y^{\card{\EA(I)}} z^{\card{\EA(B_I) \setminus \EA(I)}}.
\]
In the above, $B_I$ denotes the unique basis of $M$ such that $\IP(B_I) \subseteq I \subseteq B_I \union \EA(B_I)$, whose existence follows from Crapo's partition.  For additional discussion, see Section \ref{subsec:open-top-combinatorics}.

The rest of the paper is organized as follows.  Section 2 gives a brief overview of necessary background material on closure functions, matroids, convex geometries, and antimatroids.  Section 3 constructs the generalized external order and applies results from the theory of convex geometries to describe its structure.  Section 4 discusses matroidal closure systems, and characterizes the lattices isomorphic to the external order of an ordered matroid in terms of matroidal closure systems and supersolvable antimatroids.  In Section 5, we discuss open problems and potential directions for future research.

\section{Background}

We briefly review relevant background and notation, and refer the reader to standard sources for additional details.  We begin with notation for some of the basic objects of our discussion, set systems and closure functions.

\begin{definition}
  A \defn{set system} is a pair $(E, \script{S})$, where $E$ is a finite set and $\script{S}$ is a nonempty collection of subsets of $E$.  We will sometimes refer to $\script{S}$ as a set system when the ground set is understood.
\end{definition}

We will adopt the common notation of using a string of lower-case letters or numbers to refer to a small finite set.  For instance, if $x, y \in E$ are elements of a ground set, then the string $xy$ denotes the set $\{x, y\}$.  If $A \subseteq E$, then $A \union xy$ denotes the set $A \union \{x, y\}$.

If $A$ is a set and $P$ is a logical statement valid for the elements of $A$, then $A_P$ denotes the set $\set{x \in A}{P \text{ holds for } x}$.  For instance, if $A$ consists of elements ordered by an ordering $\leq$ and $y \in A$, then $A_{\leq y}$ denotes the set $\set{x \in A}{x \leq y}$.

\begin{definition}
  Let $E$ be a finite set, and let $\clfn : 2^E \to 2^E$.  Then $\clfn$ is called a \defn{closure function} if, for every $A, B \subseteq E$, it is:
  \begin{itemize}
    \item Extensive: $A \subseteq \clfn(A)$.
    \item Increasing: if $A \subseteq B$, then $\clfn(A) \subseteq \clfn(B)$.
    \item Idempotent: $\clfn(\clfn(A)) = \clfn(A)$.
  \end{itemize}
  A set $A \subseteq E$ is called \defn{closed} or \defn{convex} with respect to $\clfn$ if $\clfn(A) = A$.
\end{definition}

Closely related to closure functions is a class of set systems called \emph{Moore families}.

\begin{definition}
  Let $(E, \K)$ be a set system.  Then $\K$ is called a \defn{Moore family} if $\K$ contains $E$ and is closed under set intersections.
\end{definition}

Closure functions and Moore families for a set $E$ are equivalent under the following correspondence.  For a closure function $\clfn$, the collection $\K_{\clfn}$ of closed sets of $\clfn$ gives a Moore family, and inversely, for a Moore family $\K$, the mapping
\[
  \clfn_{\K} : A \mapsto \Intersect \set{K \in \K}{A \subseteq K}
\]
is a closure function.

A tuple $(E, \K, \clfn)$ is called a \defn{closure system} if $\K$ is a Moore family on the finite set $E$, and $\clfn = \clfn_{\K}$.  Often we will refer to this tuple and the underlying Moore family $\K$ interchangeably, in which case $E$ and $\clfn$ will denote the underlying finite set and corresponding closure function.

We call a closure system $\K$ \defn{reduced} if the set $K_0 = \clfn(\emptyset) = \Intersect_{K \in \K} K$ is empty.  If $\K$ is not reduced, then the Moore family $\set{K \setminus K_0}{K \in \K}$ is a reduced closure system which is structurally equivalent to $\K$.  For convenience we will assume from now on that all closure systems are reduced, unless noted otherwise.  Results stated in these terms generally are simple to extend to the non-reduced case.

The partial ordering of a closure system $\K$ under set inclusion forms a \emph{lattice}, with meet and join given by $K \meet K' = K \intersect K'$ and $K \join K' = \clfn(K \union K')$.  Throughout, we will assume familiarity with partial orders and lattices, as developed in \cite{stanley_enumerative_2011} Chapter 3.

If $A \subseteq E$, a point $x \in A$ is called an \defn{extreme point} of $A$ if $x \notin \clfn(A \setminus x)$, and the set of such points is denoted $\ex(A)$.  The extreme point map is in general idempotent, but not necessarily increasing.  We call a set $A$ an \defn{independent} set if $\ex(A) = A$, and we write $\indep(\K)$ for the collection of independent sets of $\K$, which in particular is closed under taking subsets.  A set system closed under taking subsets in this way is called a \defn{simplicial complex}.

\subsection{Matroids}

We review briefly the most relevant background on the topic of matroids, but assume general familiarity with the theory, including the definitions and relations between various cryptomorphic axiom systems.  For a comprehensive development, see \cite{oxley_matroid_2011}, and for background on matroid activity, see \cite{bjorner_homology_1992}.

\begin{definition}
  A set system $M = (E, \indep)$ is called a \defn{matroid} if:
  \begin{itemize}
    \item For every $I \in \indep$, if $J \subseteq I$, then $J \in \indep$.
    \item For every $I, J \in \indep$, if $\card{I} > \card{J}$, then there exists $x \in I \setminus J$ such that $J \union x \in \indep$.
  \end{itemize}
  A set in $\indep$ is called an \defn{independent set} of the matroid $M$.
\end{definition}

The above conditions are respectively called the \defn{hereditary} and \defn{independence} axioms for matroids, and by the first of these, the independent sets of a matroid form a simplicial complex.

Let $M = (E, \indep)$ be a matroid.  A \defn{basis} of $M$ is a maximal independent set, and a \defn{circuit} of $M$ is a minimal dependent set.  The \defn{rank function} of $M$ is the function $r : 2^E \to \ZZ_{\geq 0}$ given by
\[
  r(A) \defeq \max \set{\card{I}}{I \in \indep, I \subseteq A}.
\]

The \defn{matroid closure function} $\cl$ of $M$ is defined by
\[
  \cl(A) \defeq \set{x \in E}{r(A \union x) = r(A)},
\]
and a set closed with respect to $\cl$ is called a \defn{flat} of $M$.  The flats of a matroid $M$ form a \emph{geometric lattice} under set inclusion, and any geometric lattice determines the lattice of flats of a unique simple matroid.

The matroid closure function in particular satisfies the \emph{Steinitz-Mac Lane exchange property}: if $F \subseteq E$ with $\cl(F) = F$ and $x, y \in E \setminus F$, then $y \in \cl(F \union x)$ implies $x \in \cl(F \union y)$.
If $\flats$ denotes the closure system associated with $\cl$, then the matroid independent sets $\indep$ can be recovered by $\indep = \indep(\flats)$.  More generally, any closure function satisfying the Steinitz-Mac Lane exchange property defines a matroid in this way, so matroids can equivalently be defined in terms of such closure functions.

For notation, we denote the ground set of a matroid $M$ by $E(M)$, and the collections of independent sets, bases, circuits, and flats of $M$ are denoted respectively by $\indep(M)$, $\bases(M)$, $\circuits(M)$, and $\flats(M)$.

A matroid $M$ is called an \defn{ordered matroid} if its ground set is equipped with a total ordering.  An ordering on the ground set allows the definition of the important notion of \emph{matroid activity}.  Classically this is defined for matroid bases using \emph{fundamental circuits}, but we will use a generalization to arbitrary subsets of $E(M)$, described in the following form by Las Vergnas \cite{las_vergnas_active_2001}.

\begin{definition}
  Let $M$ be an ordered matroid, and let $A \subseteq E(M)$.  An element $x \in E(M)$ is called \defn{$\mathbf{M}$-active} with respect to $A$ if there is a circuit $C$ of $M$ with smallest element $x$ such that $C \subseteq A \union x$.  The set of $M$-active elements of $A$ is denoted by $\Act_M(A)$.

  An element in $E(M) \setminus A$ which is $M$-active with respect to $A$ is called \defn{externally active}, and otherwise is called \defn{externally passive}.  We denote the collections of externally active and externally passive elements by $\EA_M(A)$ and $\EP_M(A)$.
\end{definition}

\subsection{Convex Geometries}
\label{subsec:background-conv-geoms}

In light of the characterization of matroids in terms of closure functions, a matroid can be viewed as a closure system whose closure function satisfies the Steinitz-Mac Lane exchange property.  We now review the related class of closure systems, the \emph{convex geometries}, which are characterized by closure functions instead satisfying an \emph{anti-exchange} property motivated by convex hulls in Euclidean space.  For additional details and examples, see \cite{edelman_theory_1985} and its references.

\begin{definition}
  Let $(E, \K, \clfn)$ be a closure system.  The closure function $\clfn$ is called \defn{anti-exchange} if for every $K \in \K$ and all distinct points $x, y \notin K$ we have that $y \in \clfn(K \union x)$ implies $x \notin \clfn(K \union y)$.  A closure system whose closure function is anti-exchange is called a \defn{convex geometry}.
\end{definition}

Convex geometries relate to the convex hulls of points in Euclidean space in the following way: if $E$ is a finite collection of points in $\RR^n$, then the function mapping a subset of $E$ to the intersection of its convex hull with $E$ is an anti-exchange closure function.  While not all convex geometries can be presented in terms of this construction, this example provides helpful geometric intuition.  A convex geometry is called \defn{affine} if it arises from points in Euclidean space in this way.

The following proposition, which restates results from \cite{edelman_theory_1985} Theorems 2.1 and 2.2 and \cite{monjardet_duality_2001} Theorem 2, gives several equivalent characterizations of convex geometries which will be useful for our discussion of the matroid external order in Section \ref{sec:ext-order}.

\begin{proposition}
  \label{prop:convex-geometries}
  If $(E, \K, \clfn)$ is a closure system, then the following are equivalent:
  \begin{enumerate}
    \item $\clfn$ is anti-exchange.
    \item \label{p:closure-from-extreme-points} For every set $A \subseteq E$, $\clfn(A) = \clfn(\ex(A))$.
    \item \label{p:ex-of-closure} For every set $A \subseteq E$, $\ex(A) = \ex(\clfn(A))$.
    \item \label{p:ex-sigma-inverses} For every convex set $K$, $K = \clfn(\ex(K))$.
    \item \label{p:ex-injective} The restriction of $\ex$ to $\K$ is injective.
    \item \label{p:closure-interval} For every set $A \subseteq E$, $\clfn^{-1}(\clfn(A)) = [\ex(A), \clfn(A)]$.
    \item \label{p:ex-interval} For every set $A \subseteq E$, $\ex^{-1}(\ex(A)) = [\ex(A), \clfn(A)]$.
  \end{enumerate}
\end{proposition}

The following properties of convex geometries will be useful later, and are derived from the above characterizations.

\begin{proposition}
  \label{prop:conv-geom-additional}
  Let $(E, \K, \clfn)$ be a convex geometry.  Then:
  \begin{enumerate}
    \item \label{p:ex-cl-inverses} The extreme point map $\ex$ is a bijection from $\K$ to $\indep(\K)$ with inverse $\clfn$.
    \item For every $A, B \subseteq E$, $\clfn(A) = \clfn(B)$ if and only if $\ex(A) = \ex(B)$.
    \item \label{p:cg-crapo-partition} The intervals $[\ex(K), K], K \in \K$ form a partition of the Boolean lattice.
  \end{enumerate}
\end{proposition}

\begin{proof}
  Since the extreme point map is idempotent, $\ex(K)$ is independent for every $K \in \K$, and $\ex$ is well-defined as a map from $\K$ to $\indep(\K)$.
  The fact that $\ex$ is bijective with inverse $\clfn$ follows from
  Proposition \ref{prop:convex-geometries}, Parts \ref{p:ex-sigma-inverses}, \ref{p:ex-of-closure} and \ref{p:ex-injective}.

  For the second part, if $A \subseteq E$, then by Proposition \ref{prop:convex-geometries}, Parts \ref{p:closure-interval} and \ref{p:ex-interval}, the subsets with closure equal to $\clfn(A)$ are the subsets in the interval $[\ex(A), \clfn(A)]$, and this interval likewise gives the subsets with extreme points equal to $\ex(A)$.

  For the last part, note that by the above, the relation ``$A \sim B$ if $\clfn(A) = \clfn(B)$'' is an equivalence relation, and the equivalence class of a set $A$ is the interval $[\ex(A), \clfn(A)]$.  Each equivalence class contains a unique convex set, so the equivalence classes of $\sim$ are indexed by $K \in \K$.
\end{proof}

Convex geometries additionally can be characterized in terms of their lattices of convex sets under set inclusion.

\begin{definition}
  Let $L$ be a lattice.  Then $L$ is called \defn{meet-distributive} if whenever $v \in L$ and $u$ is the meet of all elements covered by $v$, then the interval $[u, v]$ is isomorphic to a Boolean lattice.
\end{definition}

\begin{proposition}[\cite{edelman_meet-distributive_1980}, Thm.~3.3]
  \label{prop:meet-dist-lat}
  A finite lattice $L$ is isomorphic to the lattice of convex sets of a convex geometry under set inclusion if and only if $L$ is meet-distributive.
\end{proposition}

In particular, given a meet-distributive lattice $L$, there is a canonical construction for a corresponding reduced convex geometry with ground set given by the \emph{join-irreducible} elements of $L$.  In Section \ref{sec:ext-ord-char}, we will sometimes refer to the convex sets of a convex geometry in lattice theoretic terms reflecting this correspondence.
See \cite{edelman_meet-distributive_1980} for additional details on meet-distributive lattices and their relation to convex geometries.

\subsection{Antimatroids}

We will additionally need background on the topic of \emph{antimatroids}, a class of greedoids which are essentially equivalent to convex geometries.  For a more extensive overview, see \cite{bjorner_introduction_1992} Section 8.7, and for a different perspective, see \cite{dietrich_matroids_1989}.

\begin{definition}
  A set system $(E, \F)$ is called an \defn{antimatroid} if:
  \begin{itemize}
    \item $\emptyset \in \F$.
    \item For every $A, B \in \F$, if $B \nsubseteq A$, then there is an $x \in B \setminus A$ such that $A \union x \in \F$.
  \end{itemize}
  The sets in $\F$ are called its \defn{feasible sets}.
\end{definition}

\begin{proposition}[\cite{bjorner_introduction_1992}, Prop.~8.7.3]
  A set system $(E, \F)$ is an antimatroid if and only if the complementary set system $\set{E \setminus F}{F \in \F}$ is a convex geometry.
\end{proposition}

Correspondingly, a finite lattice is the lattice of feasible sets of an antimatroid if and only if its order dual is meet-distributive.  Such lattices are called \defn{join-distributive}.

We will in particular make use of two important structures of antimatroids, their \emph{independent sets} and their \emph{rooted circuits}.  For a set system $(E, \F)$ and $A \subseteq E$, define the \defn{trace} $\trace{\F}{A}$ of $A$ in $\F$ to be the collection $\set{ F \intersect A }{ F \in \F }$.

\begin{definition}
  Let $(E, \F)$ be an antimatroid.  A set $A \subseteq E$ is called \defn{independent} if $\trace{\F}{A} = 2^A$.  A set which is not independent is called \defn{dependent}, and a minimal dependent set is called a \defn{circuit}.  We denote the collection of independent sets of $\F$ by $\indep(\F)$.
\end{definition}

The definition of independent sets for an antimatroid is equivalent to that of its corresponding convex geometry.

\begin{proposition}[\cite{bjorner_introduction_1992}, Prop.~8.7.9]
  Let $(E, \F)$ be an antimatroid, and let $\ex$ be the extreme point function of its corresponding convex geometry $\K$.  Then $A \subseteq E$ satisfies $\trace{\F}{A} = 2^A$ if and only if $A = \ex(K)$ for some convex set $K \in \K$.  In particular, $\indep(\F) = \indep(\K)$.
\end{proposition}

If $A$ is a set and $x \in A$, the tuple $(A, x)$ is called a \defn{rooted set} with \defn{root} $x$.  The following describes how any circuit of an antimatroid can be assigned a canonical root.

\begin{proposition}[\cite{bjorner_introduction_1992}, Sec.~8.7.C]
  \label{prop:rooted-circuits-def}
  Let $(E, \F)$ be an antimatroid and let $C \subseteq E$ be a circuit of $\F$.  Then there is a unique element $x \in C$ such that $\trace{\F}{C} = 2^C \setminus \{\{x\}\}$.  We call the rooted set $(C, x)$ a \defn{rooted circuit} of $\F$, and denote the rooted circuits of $\F$ by $\circuits(\F)$.
\end{proposition}

The rooted circuits of an antimatroid in particular determine its feasible sets in the following way.

\begin{proposition}[\cite{bjorner_introduction_1992}, Prop.~8.7.11]
  \label{prop:rt-circuits-characterize}
  Let $(E, \F)$ be an antimatroid and $A \subseteq E$.  Then $A$ is feasible if and only if $C \intersect A \neq \{x\}$ for every rooted circuit $(C, x)$.
\end{proposition}

\section{Convex Geometry of the Active Closure Function}
\label{sec:ext-order}

Let $M$ be an ordered matroid.  The following set functions will play a central role in the main construction of this paper, the generalized external order.  We will usually work with only a single underlying matroid, so we will often streamline our notation by omitting the matroid $M$, as long as there is no risk of confusion.

\begin{definition}
  Let $M$ be an ordered matroid.  For $A \subseteq E(M)$, define the \defn{external active closure function} $\clExt^M$ by
  \[
    \clExt^M : A \mapsto A \union \EA_M(A) = A \union \Act_M(A),
  \]
  and let the \defn{external active extreme point function} $\exExt^M$ map $A$ to its lexicographically maximal spanning independent subset.
\end{definition}

Las Vergnas defined $\clExt$ in \cite{las_vergnas_active_2001}, and showed that it is a closure function.  We begin by proving Theorem \ref{thm:ext-closure-anti-exchange}, which refines this characterization and connects the matroid active orders with the theory of finite convex geometries.

\begin{customthm}{\ref{thm:ext-closure-anti-exchange}}
  The map $\clExt$ is an anti-exchange closure function, and its corresponding extreme point function is $\exExt$.
\end{customthm}

\begin{proof}
  As noted, $\clExt$ is shown to be a closure function in \cite{las_vergnas_active_2001} Proposition 2.2.  To see that it is anti-exchange, let $A \subseteq E(M)$, and suppose that $x, y \in E(M)$ with $x \neq y$ and $x, y \notin \clExt(A)$.  In particular, this implies $x, y \notin A$.

  If $y \in \clExt(A \union x)$, then $y$ must be $M$-active with respect to $A \union x$, and so $y$ is the smallest element of some circuit $C \subseteq A \union xy$.  However, we know that $y \notin \Act(A)$ since $y \notin \clExt(A)$, so we see that $C \nsubseteq A \union y$.  We conclude that $x \in C$, and since $y$ is the smallest element of $C$, that $y < x$ in the ordering of $M$.

  If it were also true that $x \in \clExt(A \union y)$, then by the same argument we could show $x < y$.  We can thus conclude $x \notin \clExt(A \union y)$, and so $\clExt$ is anti-exchange.


  Now let $A \subseteq E(M)$, and recall that the extreme point function of $\clExt$ is defined by
  \[
    \ex(A) \defeq \set{x \in A}{x \notin \clExt(A \setminus x)}.
  \]
  Let $I = \exExt(A)$, the lexicographically maximal spanning independent subset of $A$.  Equivalently,
  \[
    I = \set{x \in A}{x \notin \cl( A_{>x} )}.
  \]

  Let $x \in A$, and suppose first that $x \notin \ex(A)$.  Then $x \in \clExt(A \setminus x)$, so $x$ is $M$-active with respect to $A \setminus x$, and thus $x$ is the smallest element of a circuit $C \subseteq A$.  In particular, $x \in \cl(C \setminus x)$ where $C \setminus x \subseteq A_{>x}$, and this implies that $x \notin I$.

  Now suppose $x \notin I$.  Then $x \in \cl( A_{>x} )$, so if $J \subseteq A_{>x}$ is a minimal subset with $x \in \cl(J)$, then $J \union x$ is a circuit in $A$ with smallest element $x$.  Thus $x$ is $M$-active with respect to $A \setminus x$, so $x \in \clExt(A \setminus x)$, and this implies that $x \notin \ex(A)$.
  We conclude $I = \ex(A)$ as desired.
\end{proof}

In particular, by application of Theorem \ref{thm:ext-closure-anti-exchange} and Proposition \ref{prop:convex-geometries}, we conclude the following.
\begin{corollary}
  \label{cor:ext-order-structure}
  If $M$ is an ordered matroid and $\K(M)$ is the closure system defined by $\clExt$, then:
  \begin{enumerate}
    \item \label{p:ext-order-conv-geom} $\K(M) = \set{I \union \EA(I)}{I \in \indep(M)}$ is a convex geometry on ground set $E(M)$.
    \item \label{p:ind-closed-bij} The independent sets of $\K(M)$ are $\indep(M)$, and are in bijection with the closed sets $\K(M)$ by the map $\clExt$.
    \item \label{p:convex-intersections} If $I, I' \in \indep(M)$, then there exists $J \in \indep(M)$ with
    \[
      J \union \EA(J) = (I \union \EA(I)) \intersect (I' \union \EA(I')).
    \]
  \end{enumerate}
\end{corollary}

\begin{proof}
  For Part \ref{p:ext-order-conv-geom}, note that the set $I \union \EA(I)$ for $I \in \indep(M)$ is the external active closure of $I$, and thus is convex in $\K(M)$.
  Likewise, if $K$ is convex, then by Proposition \ref{prop:convex-geometries}, Part \ref{p:closure-from-extreme-points}, $K = \clExt(K) = \clExt( \exExt(K) )$ is the image of the independent set $\exExt(K) \in \indep(M)$.

  For Part \ref{p:ind-closed-bij}, $\clExt$ is a bijection from $\indep(\K(M))$ to $\K(M)$ by Proposition \ref{prop:conv-geom-additional}, Part \ref{p:ex-cl-inverses}.
  In particular, $\indep(\K(M))$ is given by $\exExt(2^{E(M)})$, which is equal to $\indep(M)$ as desired.

  Part \ref{p:convex-intersections} follows from Part \ref{p:ext-order-conv-geom} and the fact that $\K(M)$ is a closure system.
\end{proof}

We can now define the \emph{external order} on the independent sets of an ordered matroid, which in particular extends the classical external order on matroid \emph{bases}, as defined in \cite{las_vergnas_active_2001}.  To simplify notation when working with the external order and its associated convex geometry, we will use $\clfn$ and $\ex$ to denote $\clExt$ and $\exExt$.  As usual, $\cl$ will denote the standard matroid closure function.

\begin{definition}
  If $M$ is an ordered matroid, define the \defn{external order} $\Ext(M)$ as the partial order $(\indep(M), \leqext)$, where $I \leqext J$ if and only if $I \union \EA(I) \subseteq J \union \EA(J)$.
\end{definition}

\begin{theorem}
  \label{thm:ext-order-lattice}
  Let $M$ be an ordered matroid.  Then the external order $\Ext(M)$ on $\indep(M)$ is a meet distributive lattice.  For every $I, J \in \indep(M)$:
  \begin{itemize}
    \item $I \meet J = \ex( \clfn(I) \intersect \clfn(J) )$, and $\clfn(I \meet J) = \clfn(I) \intersect \clfn(J)$.
    \item $I \join J = \ex( \clfn(I) \union \clfn(J) )$, and $\clfn(I \join J) \subseteq \clfn(I) \union \clfn(J)$.
  \end{itemize}
\end{theorem}

\begin{proof}
  The closure function $\clfn : \Ext(M) \to \K_M$ is a poset isomorphism, hence the poset structure of $\Ext(M)$ is equivalent to that of the lattice $\K_M$ of $\clfn$-closed sets under inclusion.  Meet-distributivity follows from Proposition \ref{prop:meet-dist-lat}.
  The descriptions of the meet and join operators follow directly from the properties of $\K_M$ as a convex geometry, noting that Proposition \ref{prop:convex-geometries}, Part \ref{p:ex-injective} is used to simplify the join operation.
\end{proof}

The external order additionally behaves well with respect to other natural orders.

\begin{proposition}
  \label{prop:ext-order-others}
  Let $M$ be an ordered matroid.  For every $I, J \in \indep(M)$:
  \begin{itemize}
    \item If $I \subseteq J$, then $I \leqext J$.
    \item If $I \leqext J$, then $\cl(I) \subseteq \cl(J)$.
    \item If $I \leqext J$, then $I \leq_{\lex} J$ in $\leq^*$-lexicographic order.
  \end{itemize}
\end{proposition}

\begin{proof}
  The first two parts are immediate by the increasing property of closure functions, and the fact that $A$ spans $\clfn(A)$ for every set $A$.

  For the last part, suppose $I \leqext J$.  Since $J = \ex(\clfn(J))$, $J$~is the $\leq$-lexicographically maximal spanning independent subset of $\clfn(J)$, and in particular can be formed by taking $\leq^*$-greedy extensions of independent sets in $\clfn(J)$.  Since $I \subseteq \clfn(I) \subseteq \clfn(J)$ is an independent subset of $\clfn(J)$, this implies that $J$ is $\leq^*$-lexicographically at least as large as $I$.
\end{proof}

In \cite{las_vergnas_active_2001}, Las Vergnas defines the \emph{external/internal order} on the bases of an ordered matroid as a suitable join of the external and internal active orders.  It would be interesting if a natural generalization of the external/internal order can be defined which incorporates the structure of $\Ext(M)$ and the dual internal order on coindependent sets.  A fundamental difficulty in producing such a construction is the fact that matroid duality and the notions of duality most suitable for convex geometries are, while related, not entirely compatible.


A direct consequence of the definition of the external order is that it is consistent with the structure of the \emph{broken circuit complex}, a construction in topological combinatorics which has been used to study important combinatorial and homological properties of matroids.  Brylawski \cite{brylawski_broken-circuit_1977} gives an overview.  A \emph{broken circuit} of an ordered matroid $M$ is a set of the form $C \setminus \min(C)$ for $C$ a circuit of $M$, and the broken circuit complex of $M$ is defined as the collection of sets containing no broken circuit.  In particular, these are the sets $I \in \indep$ with $\clExt(I) = I$, whence we conclude the following.
\begin{proposition}
  \label{prop:bc-complex-embedding}
  Let $M$ be an ordered matroid, and let $BC(M)$ be its broken circuit complex.  Then $(BC(M), \subseteq)$ is a subposet of $(\indep(M), \leqext)$ forming a lower order ideal, and the facets of $BC(M)$ are the $\leqext$-minimal bases of $M$.
\end{proposition}

We now prove Theorem \ref{thm:crapo-type-partition}, which gives an application in the direction of topological combinatorics.
Recall that for an ordered matroid $M$, a classical result of Crapo gives a partition of the Boolean lattice $2^{E(M)}$ into the Boolean subintervals $[\IP(B), B \union \EA(B)]$ for bases $B$, where $\IP(B)$ denotes the set of \emph{internally passive} elements of $B$.\footnote{See for instance \cite{bjorner_homology_1992}, Proposition~7.3.6.}
By specializing a similar partition result in the theory of convex geometries to $\K(M)$, we obtain a refinement of the Crapo partition.

\begin{customthm}{\ref{thm:crapo-type-partition}}
  Let $M$ be an ordered matroid.  Then the intervals
  \[
    [I, I \union \EA(I)], \hspace{2mm} I \in \indep(M)
  \]
  form a partition of the Boolean lattice $2^{E(M)}$, and this partition is a refinement of the classical matroid partition of Crapo.
\end{customthm}

\begin{proof}
  Applying Proposition \ref{prop:conv-geom-additional}, Part \ref{p:cg-crapo-partition} to $\K(M)$ gives that the intervals $[I, I \union \EA(I)]$, $I \in \indep(M)$, form a partition of $2^{E(M)}$.
  To see that the partition is a refinement of the partition of Crapo, note that if $I$ is in $[\IP(B), B \union \EA(B)]$ for some $B \in \bases(M)$, then $I \union \EA(I)$ is as well by the closure function properties of $\clExt$.
\end{proof}

\section{Characterizing the External Order}
\label{sec:ext-ord-char}

We turn now to characterizing the meet-distributive lattices arising from the external order of an ordered matroid.  The characterization incorporates two main ideas.

First, we define \textit{matroidal\/} meet-distributive lattices by introducing a lattice property which is equivalent to the extreme point sets being the independent sets of a matroid.  Second, we will show that \textit{supersolvability\/} ensures a type of order consistency needed for such a convex geometry to be induced by a total order on the ground set.

The following result compares matroid closure functions to other closure functions sharing the same collection of independent sets.

\begin{lemma}
  \label{lem:matroid-closures-big}
  Let $M$ be a matroid, and let $\clfn$ be a closure function with independent sets $\indep(M)$.  Then for $A \subseteq E$, $\clfn(A) \subseteq \cl(A)$.
\end{lemma}

\begin{proof}
  First let $I \in \indep(M)$, and suppose there is an $x \in \clfn(I) \setminus \cl(I)$.  Since $x \notin \cl(I)$, the set $J = I \union x$ is independent, and we have $I \subseteq J \subseteq \clfn(I)$.  In particular,
  \[
    \clfn(I) \subseteq \clfn(J) \subseteq \clfn(\clfn(I)) = \clfn(I),
  \]
  and we have $\clfn(I) = \clfn(J)$.  But since $J$ is independent, $x$ is an extreme point of $J$, hence $x \notin \clfn(J \setminus x) = \clfn(I)$, a contradiction.  We conclude $\clfn(I) \subseteq \cl(I)$ for $I$ independent.

  For the general case, recall that the closure of a set $A$ in a closure system is given by the intersection of the closed sets containing $A$.  Thus it is sufficient to prove that the flats of $M$ are closed with respect to $\clfn$.

  If $F$ is a flat of $M$, let $\beta \subseteq F$ be minimal with $\clfn(\beta) = \clfn(F)$.  For any $b \in \beta$, we have $\clfn(\beta \setminus b) \neq \clfn(\beta) = \clfn(F)$ by minimality of $\beta$, so $b \notin \clfn(\beta \setminus b)$, and $b$ is an extreme point of $\beta$ with respect to $\clfn$.

  In particular, $\ex(\beta) = \beta$, so $\beta$ is independent for $\clfn$, and thus $\beta \in \indep(M)$.  Then by the previous argument for independent sets,
  \[
    \clfn(F) = \clfn(\beta) \subseteq \cl(\beta) \subseteq \cl(F) = F \subseteq \clfn(F).
  \]
  Thus $F$ is closed with respect to $\clfn$.
\end{proof}

Before proceeding, we will need a few additional definitions.  If $\K$ is a closure system, we call $\K$ \defn{matroidal} if $\indep(\K) = \indep(M)$ for some matroid $M$.  If $L$ is a finite lattice, the \defn{covering rank function} $\rcov$ of $L$ is defined by
\[
  \rcov(x) \defeq \ncard{\set{y \in L}{x \text{ covers } y \text{ in } L}}.
\]
Note that if $L$ is the meet-distributive lattice of convex sets of a convex geometry $\K$, then the sets covered by a convex set $K$ in $L$ correspond with the extreme points of $K$, hence $\rcov(K) = \card{\ex(K)}$.
Additionally, if $\rho: L \to \RR$, then $\rho$ is called \defn{submodular} if for $x, y \in L$, $\rho$ satisfies the submodular inequality
\[
  \rho(x \join y) + \rho(x \meet y) \leq \rho(x) + \rho(y).
\]

We can now prove the following lattice-theoretic characterization of matroidal convex geometries.

\begin{theorem}
  \label{thm:matroidal-rk-cov-fn}
  Let $(E, \K)$ be a convex geometry with meet-distributive lattice $L$ of convex sets.  Then $\K$ is matroidal if and only if the rank covering function of $L$ is increasing and submodular.
\end{theorem}

\begin{proof}
  Suppose first that the rank covering function $\rcov$ of $L$ is increasing and submodular.  For $A \subseteq E$ let
  \[
    r(A) \defeq \max \set{\card{I}}{I \in \indep(\K), I \subseteq A}.
  \]
  We will show that $r$ is a matroid rank function with corresponding independent sets $\indep(\K)$.

  It follows immediately from the definition of $r$ that $0 \leq r(A) \leq \card{A}$ for $A \subseteq E$, and that $r(A) \leq r(B)$ for $A \subseteq B \subseteq E$.  It thus remains to prove that $r$ is submodular on $2^E$.  We first show that $r(A) = \card{\ex(A)}$ for $A \subseteq E$, and thus that $r(A) = \rcov(\clfn(A))$.

  Let $A \subseteq E$ and let $I$ be an independent subset of $A$.  Since $\rcov$ is increasing on $L$ and $\clfn(I) \subseteq \clfn(A)$, we have that $\card{I} = \rcov(\clfn(I)) \leq \rcov(\clfn(A)) = \card{\ex(A)}$.
  Thus $\card{\ex(A)}$ is of maximal cardinality among independent subsets of $A$, so $r(A) = \card{\ex(A)}$.

  Now let $A, B \subseteq E$, and note that $\clfn(A \intersect B) \subseteq \clfn(A) \meet \clfn(B)$ and $\clfn(A \union B) \subseteq \clfn(A) \join \clfn(B)$.  Applying the submodular inequality for $\rcov$, we have
  \begin{align*}
    r(A \intersect B) + r(A \union B)
    & = \rcov(\clfn(A \intersect B)) + \rcov(\clfn(A \union B)) \\
    & \leq \rcov(\clfn(A) \meet \clfn(B)) + \rcov(\clfn(A) \join \clfn(B)) \\
    & \leq \rcov(\clfn(A)) + \rcov(\clfn(B)) \\
    & = r(A) + r(B).
  \end{align*}

  Thus $r$ also satisfies the submodular inequality.  Finally, note that $A \subseteq E$ satisfies $r(A) = \card{A}$ if and only if $A \in \indep(\K)$, so the matroid defined by rank function $r$ has independent sets $\indep(\K)$.


  For the converse, suppose $\K$ is matroidal.  Let $M$ be the matroid with $\indep(\K) = \indep(M)$, and let $r$ be its rank function.  We first show that $r$ is equal to $\rcov$ for convex sets.
  If $K \in \K$, we see by Lemma \ref{lem:matroid-closures-big} that $\cl(\ex(K)) \supseteq \clfn(\ex(K)) = \clfn(K) \supseteq K$, and thus $\ex(K)$ spans $K$ in $M$.
  Since $\ex(K)$ is additionally independent in $M$, this implies that $\card{\ex(K)}$ gives the rank of $K$ in $M$, hence $\rcov(K) = r(K)$.

  Now let $K, K' \in \K$.  Applying equality of $r$ and $\rcov$ for convex sets, we additionally have $\rcov(K \meet K') = r(K \intersect K')$ and $\rcov(K \join K') = r(K \union K')$, the latter following from Proposition \ref{prop:convex-geometries}, Part \ref{p:ex-of-closure}.
  We can then conclude that $\rcov$ is increasing and submodular on $L$ by the corresponding properties of the matroid rank function~$r$.
\end{proof}

By this result, we see that the external order for an ordered matroid has increasing, submodular covering rank function.  On the other hand, not every matroidal convex geometry comes about in this way, as the following example demonstrates.

\begin{example}
  \label{ex:matroidal-not-ext-order}
  Consider the convex geometry on ground set $E = \{a, b, c, d\}$ whose convex sets are $\K =  \{\emptyset, a, b, c, d, ab, ac, bd, abc, abd, abcd\}$.  The Hasse diagram for the corresponding meet-distributive lattice appears in Figure \ref{fig:eo-mmd-counterexample}.

  In particular, the collection of independent sets of this convex geometry is the uniform matroid $U^2_4$ of rank 2 on 4 elements.  Suppose this were the external order with respect to some total ordering $<$ on $E$.  In this case, we observe that
  \begin{itemize}
    \item $a$ is active with respect to $I = bc$, so $a$ is smallest in the unique circuit $abc$ in $I \union a$
    \item $b$ is active with respect to $J = ad$, so $b$ is smallest in the unique circuit $abd$ in $J \union b$
  \end{itemize}
  But this implies that both $a < b$ and $b < a$, a contradiction.  Thus this lattice cannot be realized as the external order of $U^2_4$ with respect to any total ordering on $E$.
\end{example}

\begin{figure}[ht]
  \centering
  \begin{minipage}{.4\textwidth}
    \centering
    \begin{tikzpicture}[yscale=1.25, xscale=0.75]
      \tikzstyle{feas set style}=[inner sep=1mm, shape=rectangle]
      \tikzstyle{edge label style}=[midway, inner sep=0.5mm, shape=rectangle, fill=white]
      \newcommand*{\leftpad}{0mm}

      \node[feas set style] (pcd) at ( 0,4) {abcd};
      \node[feas set style] (pbc) at (-1,3) {abc};
      \node[feas set style] (pad) at ( 1,3) {abd};
      \node[feas set style] (pac) at (-2,2) {ac};
      \node[feas set style] (pab) at ( 0,2) {ab};
      \node[feas set style] (pbd) at ( 2,2) {bd};
      \node[feas set style] (ppc) at (-3,1) {c};
      \node[feas set style] (ppa) at (-1,1) {a};
      \node[feas set style] (ppb) at ( 1,1) {b};
      \node[feas set style] (ppd) at ( 3,1) {d};
      \node[feas set style] (ppp) at ( 0,0) {$\emptyset$};

      \draw (pcd) -- (pbc) node [edge label style] {\tiny d};
      \draw (pcd) -- (pad) node [edge label style] {\tiny c};
      \draw (pbc) -- (pac) node [edge label style] {\tiny b};
      \draw (pbc) -- (pab) node [edge label style] {\tiny c};
      \draw (pad) -- (pab) node [edge label style] {\tiny d};
      \draw (pad) -- (pbd) node [edge label style] {\tiny a};
      \draw (pac) -- (ppc) node [edge label style] {\tiny a};
      \draw (pac) -- (ppa) node [edge label style] {\tiny c};
      \draw (pab) -- (ppa) node [edge label style] {\tiny b};
      \draw (pab) -- (ppb) node [edge label style] {\tiny a};
      \draw (pbd) -- (ppb) node [edge label style] {\tiny d};
      \draw (pbd) -- (ppd) node [edge label style] {\tiny b};
      \draw (ppc) -- (ppp) node [edge label style] {\tiny c};
      \draw (ppa) -- (ppp) node [edge label style] {\tiny a};
      \draw (ppb) -- (ppp) node [edge label style] {\tiny b};
      \draw (ppd) -- (ppp) node [edge label style] {\tiny d};
    \end{tikzpicture}

  \end{minipage}
  \begin{minipage}{.4\textwidth}
    \centering
    \begin{tikzpicture}[yscale=1.25, xscale=0.75]
      \tikzstyle{feas set style}=[inner sep=1mm, shape=rectangle]
      \tikzstyle{edge label style}=[midway, inner sep=0.5mm, shape=rectangle, fill=white]
      \newcommand*{\leftpad}{0mm}
      \newcommand*{\highlightcolor}{gray!30!white}

      \node[feas set style] (pcd) at ( 0,4) {cd};
      \node[feas set style] (pbc) at (-1,3) {bc};
      \node[feas set style] (pad) at ( 1,3) {ad};
      \node[feas set style] (pac) at (-2,2) {ac};
      \node[feas set style] (pab) at ( 0,2) {ab};
      \node[feas set style] (pbd) at ( 2,2) {bd};
      \node[feas set style] (ppc) at (-3,1) {c};
      \node[feas set style] (ppa) at (-1,1) {a};
      \node[feas set style] (ppb) at ( 1,1) {b};
      \node[feas set style] (ppd) at ( 3,1) {d};
      \node[feas set style] (ppp) at ( 0,0) {$\emptyset$};

      \draw (pcd) -- (pbc);
      \draw (pcd) -- (pad);
      \draw (pbc) -- (pac);
      \draw (pbc) -- (pab);
      \draw (pad) -- (pab);
      \draw (pad) -- (pbd);
      \draw (pac) -- (ppc);
      \draw (pac) -- (ppa);
      \draw (pab) -- (ppa);
      \draw (pab) -- (ppb);
      \draw (pbd) -- (ppb);
      \draw (pbd) -- (ppd);
      \draw (ppc) -- (ppp);
      \draw (ppa) -- (ppp);
      \draw (ppb) -- (ppp);
      \draw (ppd) -- (ppp);
    \end{tikzpicture}
  \end{minipage}

  \caption{Convex sets of $\K$ with edge labels, and corresponding independent sets}
  \label{fig:eo-mmd-counterexample}
\end{figure}


To bridge the gap between matroidal convex geometries and the external order, we will need one more key notion, a characterization of supersolvable join-distributive lattices in terms of their corresponding antimatroid feasible sets, proven by Armstrong \cite{armstrong_sorting_2009}, Section 2.
We will not need additional background on supersolvable lattices beyond this characterization, but we refer the reader to \cite{stanley_supersolvable_1972} for more details.

\begin{definition}
  A set system $(E, \F)$ is called a \defn{supersolvable antimatroid} with respect to a total ordering $\leq$ on $E$ if:
  \begin{itemize}
    \item $\emptyset \in \F$.
    \item For every $A, B \in \F$, if $B \nsubseteq A$ and $x = \min_{\leq}(B \setminus A)$, then $A \union x \in \F$.
  \end{itemize}
\end{definition}

It can be seen that a supersolvable antimatroid is in particular an antimatroid.  In \cite{armstrong_sorting_2009} Theorem 2.13, Armstrong relates this property to lattice supersolvability as follows.

\begin{proposition}
  \label{prop:ss-antimat-ss-lat}
  Let $(E, \F)$ be an antimatroid with join-distributive lattice $L$ of feasible sets.  Then $L$ is a supersolvable lattice if and only if there exists a total ordering on $E$ with respect to which $\F$ is a supersolvable antimatroid.
\end{proposition}

We give an additional equivalent characterization for supersolvable antimatroids in terms of antimatroid rooted circuits.

\begin{theorem}
  \label{thm:ss-antimat-circuit-order}
  Let $(E, \F)$ be an antimatroid, and let $\leq$ be a total ordering on $E$.  Then $\F$ is supersolvable with respect to $\leq$ if and only if for every rooted circuit $(C, r)$ of $\F$, the $\leq$-maximal element of $C$ is $r$.
\end{theorem}

\begin{proof}
  Suppose first that any rooted circuit $(C, r)$ of $\F$ has root $r = \max(C)$.  Let $A, B \in \F$ with $B \nsubseteq A$, and let $x = \min(B \setminus A)$.  Suppose $A \union x \notin \F$.
  Then by Proposition \ref{prop:rt-circuits-characterize}, there exists a rooted circuit $(C, r)$ with $C \intersect (A \union x) = r$.
  Since $A$ is feasible, $C \intersect A \neq r$, so this intersection must be empty, and in particular, the root of $C$ is $x$.

  Since $B$ is feasible, $C \intersect B \neq x$, so $C \intersect B$ contains an element $y \neq x$.  By assumption, $x = \max(C)$, so $y < x$.  On the other hand, since $x = \min(B \setminus A)$, we see that $y \in A$.  But this implies that $y \in C \intersect A$, contradicting $C \intersect A = \emptyset$.  We conclude that $A \union x \in \F$, and hence that $\F$ is supersolvable with respect to $\leq$.

  For the converse, suppose $\F$ is supersolvable with respect to $\leq$, and let $(C, r)$ be a rooted circuit of $\F$ with $x = \max(C)$.  Suppose for a contradiction that $x \neq r$.
  Let $A = \Union \set{F \in \F}{F \intersect C = \emptyset}$, and let $B \in \F$ such that $B \intersect C = \{x, r\}$, which exists by Proposition \ref{prop:rooted-circuits-def}.

  Let $y = \min(B \setminus A)$, so that by supersolvability, $A \union y \in \F$.  If $y < r$, then $y \notin C$ since $r$ is the smallest element of $B \intersect C$.  In particular, this then contradicts maximality of $A$ among feasible sets avoiding $C$, since $C \intersect (A \union y) = \emptyset$.    Thus we must have $y = r$.  But then $C \intersect (A \union r) = r$, which implies $A \union r \notin \F$, a contradiction.
  We conclude that the root of $C$ must coincide with its $\leq$-maximal element, and this holds for all rooted circuits of $\F$.
\end{proof}

The families of rooted sets which are the rooted circuits of an antimatroid can be characterized axiomatically as described in \cite{bjorner_introduction_1992}, Theorem 8.7.12.  In this reference, the authors note that the axioms bear a curious resemblance to the circuit axioms for matroids.  The following lemma gives an explanation for this resemblance.

\begin{lemma}
  \label{lem:ext-antimat-circuit-roots}
  Let $M$ be an ordered matroid, and let $\F$ be the antimatroid with feasible sets $\set{E(M) \setminus K}{K \in \K_M}$.  Then the rooted circuits of $\F$ are given by
  \[
    \circuits = \set{(C, r)}{C \in \circuits(M), r = \min(C)}.
  \]
\end{lemma}

\begin{proof}
  The (non-rooted) circuits of $\F$ are its minimal dependent sets, which in particular are given by $\circuits(M)$ since $\F$ has independent sets $\indep(M)$.

  Let $C \in \circuits(M)$, and let $x = \min(C)$.  Suppose that there is a feasible set $F \in \F$ such that $F \intersect C = \{x\}$.  If $F = E(M) \setminus K$ for $K \in \K_M$, then this implies that $x \notin K$, but $C \subseteq K \union x$.
  Since $x$ is minimal in $C$, this implies $x \in \EA(K) \subseteq \clfn(K) = K$, hence $x \in K$, a contradiction.
  Thus no feasible set intersects $C$ in $\{x\}$, and we see by Proposition \ref{prop:rooted-circuits-def} that $x$ is the circuit root of $C$ in $\F$.
\end{proof}

We can now give a proof of Theorem \ref{thm:ext-order-char}, characterizing the lattices corresponding to $\Ext(M)$ for some ordered matroid $M$.

\begin{customthm}{\ref{thm:ext-order-char}}
  A finite lattice $L$ is isomorphic to the external order $\leqext$ of an ordered matroid if and only if it is meet-distributive, supersolvable, and has increasing and submodular covering rank function.
\end{customthm}

\begin{proof}
  First suppose that $L$ is isomorphic to the external order of an ordered matroid $M$.
  By Theorem \ref{thm:ext-order-lattice}, $L$ is meet-distributive, and if $\K_M$ is the corresponding convex geometry, then by Corollary \ref{cor:ext-order-structure}, the independent sets of $\K_M$ are $\indep(M)$, so $\K_M$ is a matroidal closure system.
  By Theorem \ref{thm:matroidal-rk-cov-fn}, the covering rank function of $L$ is thus increasing and submodular.

  Let $\F_M$ be the complementary antimatroid $\set{E(M) \setminus K}{K \in \K_M}$.
  By Lemma \ref{lem:ext-antimat-circuit-roots}, the rooted circuits of $\F_M$ are given by $\set{(C, r)}{C \in \circuits(M), r = \min(C)}$,
  so by Theorem \ref{thm:ss-antimat-circuit-order}, $\F_M$ is a supersolvable antimatroid with respect to the reverse order $\leq^*$.
  This implies the lattice $L^*$ of $\F_M$ under set inclusion is supersolvable by Proposition \ref{prop:ss-antimat-ss-lat}, so since supersolvability is preserved by reversing a lattice, we conclude that $L$ is supersolvable.

  Now suppose $L$ is a finite lattice which is meet-distributive, supersolvable, and has increasing and submodular covering rank function.
  If $\K$ is the convex geometry associated with $L$ as a meet-distributive lattice, then $\indep(\K)$ is given by $\indep(M)$ for a matroid $M$ by Theorem \ref{thm:matroidal-rk-cov-fn}.  Let $\F$ be the complementary antimatroid $\set{E(M) \setminus K}{K \in \K}$.  The (non-rooted) circuits of $\F$ are the minimal dependent sets of $\F$, and thus coincide with the matroid circuits of $M$.

  Since the lattice $L^*$ of feasible sets of $\F$ is supersolvable, $\F$ is a supersolvable antimatroid with respect to a total ordering $\leq$ on $E(M)$, and by Theorem \ref{thm:ss-antimat-circuit-order}, the root of each circuit of $\F$ is given by its $\leq$-maximum, or equivalently, its $\leq^*$-minimum.  Since an antimatroid is determined by its rooted circuits, this implies that $\F$ is equal to the complementary antimatroid of the convex geometry $\K_M$ for $M$ ordered by $\leq^*$.
  Thus $\K = \K_M$, and we conclude that $L$ is isomorphic to the external order of $(M, \leq^*)$.
\end{proof}

\section{Future Work}
\label{sec:future-work}

The construction of the external order $\Ext(M)$ and its associated convex geometry $\K_M$ naturally gives rise to several potential directions for future work.

\subsection{Convex Geometries}
\label{subsec:future-work-conv-geoms}

The construction of the external order described here is given in the setting of \emph{ordered matroids}.  However, there are many variants and generalizations of matroids that exhibit similar structure, and it would be interesting to consider whether similar constructions to the external order can be carried out in such related contexts.

\begin{problem}
  Define natural convex geometry structures for additional classes of objects related to matroids.
\end{problem}

For example, several of the \emph{quasi-matroidal classes} of Samper \cite{samper_quasi-matroidal_2020,samper_relaxations_2016} admit orders corresponding to the classical active orders on matroid bases.
Similarly, in \cite{las_vergnas_active_2001} Section 7, Las Vergnas defines active orders corresponding to \emph{matroid perspectives}.
The theory of \emph{graph fourientations} generalizes the notion of orientations on graphs, and in particular admits a concept of activity related to the underlying graph Tutte polynomial; see Backman and Hopkins \cite{backman_fourientations_2017} Section 3.5 for details.

A more concrete question is to determine the values of various parameters associated with convex sets in the case of the convex geometry $\K_M$.  The \emph{Helly number} of a convex geometry is known to be the maximal size of an independent set (see for instance \cite{edelman_theory_1985} Theorem 4.6), and thus is given for $\K_M$ by the rank of $M$.  Some additional convexity parameters to determine are the Radon number, Carath\'eodory number, and convex dimension \cite{edelman_theory_1985,kay_axiomatic_1971}.

\begin{problem}
  For an ordered matroid $M$, find the Radon number, Carath\'eodory number, and convex dimension of the convex geometry $\K_M$.
\end{problem}

Recall from Section \ref{subsec:background-conv-geoms} that a convex geometry is called \emph{affine} if it arises from the convex hulls of a finite collection of points in Euclidean space.  The problem of classifying the affine convex geometries was posed in \cite{edelman_theory_1985}, and remains open with some partial results.  For an overview of current progress, see \cite{adaricheva_representation_2019}.

Restricting to the setting of the external order could potentially be more tractible than the general case.  We call an ordered matroid \defn{convex representable} if its external order defines an affine convex geometry.

\begin{problem}
  Classify the ordered matroids which are convex representable.
\end{problem}

\subsection{Matroid $h$-vectors}

Recall that the \emph{$f$-vector} of a matroid is the list $(f_0, f_1, \ldots, f_d)$ enumerating the independent sets of each size, and the \emph{$h$-vector} is an invertible transformation of the $f$-vector given by
\[
  h_i = \ncard{\set{B \in \bases(M)}{\card{\IP(B)} = i}}.
\]
The $h$-vector of a matroid $M$ is closely related to the matroid active orders: in the internal active order, the $h$-vector of $M$ can be read as the number of bases of a given rank.

A recent innovation in the study of matroid $h$-vectors is the class of \emph{internally perfect matroids} of Dall \cite{dall_internally_2017}.  These matroids are characterized in terms of certain conditions related to the internal active order, and this class of matroids is notably shown by Dall to satisfy the outstanding conjecture of Stanley that the $h$-vector of a matroid is a \emph{pure $O$-sequence}.  Studying the (dual) convexity theory of internally perfect matroids could improve our understanding of matroid $h$-vectors, and potentially shed light on the general case of Stanley's conjecture.

\subsection{Topological Combinatorics}
\label{subsec:open-top-combinatorics}

In the area of topological combinatorics, matroid activities and the active orders are richly connected with the theory of independence and broken circuit complexes, in particular through generating functions such as the shelling and Tutte polynomials.  Examining how $\Ext(M)$ relates to these classical constructions could yield some new insights.

As one example, for an ordered matroid $M$, define the \defn{external structure polynomial} $P_M$ of $M$ by
\[
  P_M(x, y, z) \defeq \sum_{I \in \indep(M)} x^{r(M) - \card{I}} y^{\card{\EA(I)}} z^{\card{\EA(B_I) \setminus \EA(I)}},
\]
where $B_I$ denotes the unique basis of $M$ such that $\IP(B_I) \subseteq I \subseteq B_I \union \EA(B_I)$.  Then $P_M$ can be seen to generalize the \emph{Tutte polynomial} of $M$ by $P_M(x, y, y) = T_M(x+1, y)$.  Another specialization of $P_M$ is given by
\[
  P_M(x, y, 1) = \sum_{F \in \flats(M)} x^{r(M) - r(F)} h_{(M|_F)^*}(y),
\]
where $h$ denotes the shelling polynomial of a matroid independence complex.
Both of these specializations of $P_M$ are in particular independent of the ordering of $M$.  However, it is not known whether $P_M$ itself is independent of the ordering of $M$, or if it admits a natural topological interpretation.

As a related problem, in \cite{ardila_topology_2016}, the authors explore connections between linear extensions of the active orders on matroid bases and shelling orders of several related shellable simplicial complexes.  It would be interesting if the linear extensions of the generalized external order could be viewed similarly.
\begin{problem}
  Determine whether the linear extensions of $\Ext(M)$ correspond with the shelling orders of an appropriate simplicial complex.
\end{problem}

\subsection*{Acknowledgments}

The author would like to thank Federico Ardila, Spencer Backman, Anders Bj\"orner, Maria Gillespie, Olga Holtz, and Jose Samper for their helpful input and guidance in the preparation of this manuscript.  The author additionally thanks the anonymous referee for valuable feedback, and in particular for the suggestion to present the results using the language of closure functions, which significantly streamlined the exposition.  The research leading to these results received funding from the National Science Foundation under agreement No.\ DMS-1303298.  Any opinions, findings and conclusions or recommendations expressed in this material are those of the author, and do not necessarily reflect the views of the National Science Foundation.



\begin{thebibliography}{99}

\bibitem{adaricheva_representation_2019}
K.~Adaricheva and M.~Bolat.
\newblock Representation of convex geometries by circles on the plane.
\newblock {\em Discrete Mathematics}, 342(3):726--746, Mar. 2019.

\bibitem{ardila_topology_2016}
F.~Ardila, F.~Castillo, and J.~A. Samper.
\newblock The topology of the external activity complex of a matroid.
\newblock {\em Electronic Journal of Combinatorics}, 23(3), 2016.

\bibitem{armstrong_sorting_2009}
D.~Armstrong.
\newblock The sorting order on a {Coxeter} group.
\newblock {\em Journal of Combinatorial Theory, Series A}, 116(8):1285--1305,
  Nov. 2009.

\bibitem{backman_fourientations_2017}
S.~Backman and S.~Hopkins.
\newblock Fourientations and the {Tutte} polynomial.
\newblock {\em Research in the Mathematical Sciences}, 4(1), Sept. 2017.

\bibitem{bari_chromatic_polynomials_1979}
R.~A. Bari.
\newblock Chromatic polynomials and the internal and external activities of
  tutte.
\newblock In {\em Graph theory and related topics (Proc.~Conf., Univ. Waterloo,
  Waterloo, Ont., 1977)}, pages 41 -- 52, New York-London, 1979. Academic
  Press.

\bibitem{bjorner_homology_preprint_1979}
A.~Bj\"orner.
\newblock Homology of matroids.
\newblock Preprint, Mittag--Leffler Institute, 1979.

\bibitem{bjorner_homology_1992}
A.~Bj\"orner.
\newblock The homology and shellability of matroids and geometric lattices.
\newblock In {\em Matroid {Applications}}, volume~40 of {\em Encyclopedia of
  {Mathematics} and {Its} {Applications}}. Cambridge University Press, 1992.

\bibitem{bjorner_introduction_1992}
A.~Bj\"orner and G.~M. Ziegler.
\newblock Introduction to greedoids.
\newblock In {\em Matroid {Applications}}, volume~40 of {\em Encyclopedia of
  Mathematics and Its Applications}. Cambridge University Press, 1992.

\bibitem{brylawski_broken-circuit_1977}
T.~Brylawski.
\newblock The {broken}-{circuit} {complex}.
\newblock {\em Transactions of the American Mathematical Society},
  234(2):417--433, 1977.

\bibitem{crapo_tutte_1969}
H.~H. Crapo.
\newblock The {Tutte} polynomial.
\newblock {\em Aequationes Mathematicae}, 3:211--229, Oct. 1969.

\bibitem{dall_internally_2017}
A.~Dall.
\newblock Internally {perfect} {matroids}.
\newblock {\em Electronic Journal of Combinatorics}, 24(2), 2017.

\bibitem{dawson_collection_1984}
J.~E. Dawson.
\newblock A collection of sets related to the tutte polynomial of a matroid.
\newblock In K.~M. Koh and H.~P. Yap, editors, {\em Graph {Theory} {Singapore}
  1983}, Lecture {Notes} in {Mathematics}, pages 193--204. Springer Berlin
  Heidelberg, 1984.

\bibitem{dietrich_matroids_1989}
B.~L. Dietrich.
\newblock Matroids and antimatroids --- a survey.
\newblock {\em Discrete Mathematics}, 78(3):223--237, Jan. 1989.

\bibitem{edelman_meet-distributive_1980}
P.~H. Edelman.
\newblock Meet-distributive lattices and the anti-exchange closure.
\newblock {\em Algebra Universalis}, 10(1):290--299, Dec. 1980.

\bibitem{edelman_lattice_1982}
P.~H. Edelman.
\newblock The lattice of convex sets of an oriented matroid.
\newblock {\em Journal of Combinatorial Theory, Series B}, 33(3):239--244, Dec.
  1982.

\bibitem{edelman_theory_1985}
P.~H. Edelman and R.~E. Jamison.
\newblock The theory of convex geometries.
\newblock {\em Geometriae Dedicata}, 1985.

\bibitem{gillespie_generalized_2018}
B.~Gillespie.
\newblock {\em The {generalized} {external} {order}, and {applications} to
  {zonotopal} {algebra}}.
\newblock PhD thesis, UC Berkeley, 2018.

\bibitem{gordon_generalized_1990}
G.~Gordon and L.~Traldi.
\newblock Generalized {activities} and the {Tutte} {polynomial}.
\newblock {\em Discrete Mathematics}, 85, 1990.

\bibitem{holtz_zonotopal_2011}
O.~Holtz and A.~Ron.
\newblock Zonotopal algebra.
\newblock {\em Advances in Mathematics}, 227(2):847--894, June 2011.

\bibitem{jamison_perspective_convexity_1982}
R.~E. Jamison.
\newblock A perspective on abstract convexity: Classifying alignments by
  varieties.
\newblock {\em Convexity and Related Combinatorial Geometry}, pages 113--150,
  1982.

\bibitem{kay_axiomatic_1971}
D.~C. Kay and E.~W. Womble.
\newblock Axiomatic convexity theory and relationships between the
  {Carath{\'e}odory}, {Helly}, and {Radon} numbers.
\newblock {\em Pacific Journal of Mathematics}, 38(2):471--485, 1971.

\bibitem{las_vergnas_convexity_1980}
M.~Las~Vergnas.
\newblock Convexity in oriented matroids.
\newblock {\em Journal of Combinatorial Theory, Series B}, 29(2):231--243, Oct.
  1980.

\bibitem{las_vergnas_active_2001}
M.~Las~Vergnas.
\newblock Active orders for matroid bases.
\newblock {\em European Journal of Combinatorics}, 22(5):709--721, July 2001.

\bibitem{lenz_zonotopal_2016}
M.~Lenz.
\newblock Zonotopal algebra and forward exchange matroids.
\newblock {\em Advances in Mathematics}, 294:819--852, May 2016.

\bibitem{monjardet_duality_2001}
B.~Monjardet and V.~Raderanirina.
\newblock The duality between the anti-exchange closure operators and the path
  independent choice operators on a finite set.
\newblock {\em Mathematical Social Sciences}, 41(2):131--150, Mar. 2001.

\bibitem{oxley_matroid_2011}
J.~G. Oxley.
\newblock {\em Matroid Theory}.
\newblock Oxford University Press, New York, 2nd edition, 2011.

\bibitem{samper_quasi-matroidal_2020}
J.~Samper.
\newblock Quasi-matroidal classes of ordered simplicial complexes.
\newblock {\em Journal of Combinatorial Theory, Series A}, 175, 2020.

\bibitem{samper_relaxations_2016}
J.~A. Samper.
\newblock Relaxations of the matroid axioms {I}: {independence}, {exchange} and
  {circuits}.
\newblock {\em Proceedings of the 28th International Conference on Formal Power
  Series and Algebraic Combinatorics}, 2016.

\bibitem{stanley_supersolvable_1972}
R.~P. Stanley.
\newblock Supersolvable lattices.
\newblock {\em Algebra Universalis}, 2(1):197, Dec. 1972.

\bibitem{stanley_enumerative_2011}
R.~P. Stanley.
\newblock {\em Enumerative Combinatorics: Volume 1}.
\newblock Cambridge University Press, New York, NY, 2nd edition, 2011.

\end{thebibliography}
\end{document}